\documentclass[a4paper, 11pt]{article}

\usepackage{amsmath, amscd, amsthm, amssymb, mathrsfs, color, framed, calc}
\usepackage[applemac]{inputenc}
\usepackage[all]{xy}
\usepackage[T1]{fontenc}
\usepackage{textcomp}
\usepackage{geometry}
\usepackage{graphicx}

\usepackage{mathptmx}

\DeclareMathOperator{\SO}{SO}

\DeclareMathOperator{\U}{U}
\DeclareMathOperator{\SU}{SU}

\DeclareMathOperator{\Lie}{Lie}

\newcommand{\R}{\mathbb R}
\newcommand{\C}{\mathbb C}
\newcommand{\Z}{\mathbb Z}

\newcommand{\so}{\mathfrak{so}}

\renewcommand{\P}{\mathbb P}
\renewcommand{\H}{\mathbb H}
\renewcommand{\O}{\mathcal O}
\renewcommand{\u}{\mathfrak{u}}

\theoremstyle{plain}
	\newtheorem{theorem}{Theorem}
	\newtheorem{proposition}[theorem]{Proposition}
	\newtheorem{lemma}[theorem]{Lemma}
	
	\newtheorem{conjecture}[theorem]{Conjecture}
    \newtheorem{question}[theorem]{Question}
	
\theoremstyle{definition}
	\newtheorem{definition}[theorem]{Definition}

\theoremstyle{plain}
	\newtheorem*{theorem*}{Theorem}
	\newtheorem*{proposition*}{Proposition}
	\newtheorem*{lemma*}{Lemma}
	\newtheorem*{corollary*}{Corollary}
	\newtheorem*{conjecture*}{Conjecture}
\theoremstyle{definition}
	\newtheorem*{definition*}{Definition}
	\newtheorem*{remark*}{Remark}
	\newtheorem*{remarks*}{Remarks}
	
\makeatletter
\def\blfootnote{\xdef\@thefnmark{}\@footnotetext}
\makeatother

\begin{document}

\title{The diversity of symplectic Calabi--Yau six-manifolds}
\author{Joel Fine and Dmitri Panov}
\date{ }

\maketitle

\begin{abstract}
Given an integer $b$ and a finitely presented group $G$ we produce a compact symplectic six-manifold with $c_1=0$, $b_2>b$, $b_3>b$ and $\pi_1 = G$. In the simply-connected case we can also arrange for $b_3=0$; in particular these examples are not diffeomorphic to K\"ahler manifolds with $c_1=0$. The construction begins with a certain orientable four-dimensional hyperbolic orbifold assembled from right-angled 120-cells. The twistor space of the hyperbolic orbifold is a symplectic Calabi--Yau orbifold; a crepant resolution of this last orbifold produces a smooth symplectic manifold with the required properties.
\end{abstract}

\section{Overview}

A symplectic Calabi--Yau manifold is a symplectic manifold with vanishing integral first Chern class. A folklore conjecture states that 4-dimensional symplectic Calabi--Yau manifolds are diffeomorphic to $K3$ surfaces or $T^2$-bundles over $T^2$. Evidence supporting this conjecture is given in \cite{B} and \cite{Li} where it is proven that a symplectic Calabi--Yau 4-manifold has the rational homology of either a $K3$-surface or a $T^2$-bundle over $T^2$. In particular, the fundamental groups of such manifolds are strongly constrained. The goal of this article is to show that in dimension $6$ the situation is drastically different:

\begin{theorem}\label{main}
Given an integer $b$ and a finitely presented group $G$ there is a compact symplectic Calabi--Yau 6-manifold with $b_2> b$, $b_3>b$ and whose fundamental group is isomorphic to $G$.
 \end{theorem}

Put briefly, the proof is as follows. We first use the fact that associated to each four-dimensional orientable hyperbolic orbifold is a symplectic six-dimensional Calabi--Yau orbifold, namely its twistor space. Moreover, the topological fundamental group of both orbifolds coincide. Next we appeal to the following theorem, proven in \cite{PP}. This work was motivated by a question due to Gromov \cite[page 12]{gromov}, asking if every compact $PL$ manifold can be obtained as a quotient of $\mathbb H^n$ by a discrete co-compact group of isometries.

\begin{theorem}[Panov--Petrunin \cite{PP}]\label{mainhyper}
Given a finitely presented group $G$, there is a compact orientable hyperbolic orbifold $\mathbb{H}^4/\Gamma$ with fundamental group (of the underlying topological space) isomorphic to $G$. 
\end{theorem}

From here, the main work required to prove Theorem \ref{main} is to show that all symplectic $6$-dimensional orbifolds generated by Theorem \ref{mainhyper} admit crepant symplectic resolutions (that do not change the fundamental group). A similar strategy was already applied in \cite{FP} to a particular simply-connected hyperbolic orbifold; this led to the first known example of a simply-connected non-K\"ahler symplectic Calabi--Yau. In the simply-connected case of Theorem~\ref{main} we can consider instead a simpler collection of orbifolds which do not come from Theorem \ref{mainhyper} but which have the same type of singularities; this leads to an infinite series of \emph{non-K\"ahler} symplectic Calabi--Yau 6-manifolds. We state this as a separate result; see \S\ref{simply} for the proof.

\begin{theorem}\label{non-Kahler_simply-connected}
Given any integer $b$ there exists a simply-connected compact symplectic Calabi--Yau manifold with $b_3=0$ and $b_2\geq b$. In particular, these symplectic manifolds are not diffeomorphic to Kähler manifolds with $c_1=0$.
\end{theorem}

We carry out the resolution of the Calabi--Yau orbifolds in \S\ref{resolution} via symplectic cutting, using the fact that the singularities are locally toric and that these local pictures are mutually compatible in a precise sense. The idea is to find a neighbourhood of the singular locus whose complement is a manifold with corners. We then collapse certain tori in the boundary of the complement to produce a smooth manifold with symplectic structure. This collapsing procedure---a kind of  ``semi-local Delzant construction''--- is introduced in \S\ref{cutting_thm}. We are hopeful it will find applications besides the one envisaged here. In \S\ref{topology} we check that the resolution has not altered the fundamental group nor the fact that $c_1=0$ and, finally, that the singularities of $\H^4/\Gamma$ can be chosen so that the resolution has arbitrarily large $b_2$ and $b_3$.

The interest in the construction of symplectic Calabi--Yau manifolds goes back at least as far as the article \cite{STY}  where several potential candidates for non-K\"ahler simply-connected symplectic Calabi--Yau $6$-manifolds were constructed, although whether these examples are genuinely not K\"ahler remains unresolved. More recently, two articles \cite{Ah,BK} inspired by \cite{STY} give simple constructions of several symplectic Calabi--Yau six-manifolds: \cite{Ah} contains examples with fundamental groups $0$ and $\mathbb Z$ while \cite{BK} contains an infinite series of examples of non-simply-connected manifolds with bounded Betti numbers. Certain symplectic Calabi-Yau six-manifolds appeared as well in \cite{TY} where the question 
of existence of symplectic six-manifolds with given fundamental group and given Chern numbers was 
studied.

Finally, we mention that a large collection of symplectic Calabi--Yau 6-manifolds with hyperbolic fundamental group was given by Reznikov, who first uncovered the link between hyperbolic and symplectic geometry in \cite{Re}. Unfortunately Reznikov's article was almost unknown to the community of symplectic geometers until his construction was rediscovered and exploited in \cite{FP0} (where it was also observed that they have $c_1=0$).

\subsubsection*{Acknowledgements}

Firstly, we would like to thank Anton Petrunin, without whom this work would not exist. We would also like to thank the Simons Center for Geometry and Physics, at Stony Brook University for providing a stimulating working environment. Dmitri Panov is supported by a Royal Society University Research Fellowship.

\section{From hyperbolic to symplectic}

\subsection{A coadjoint orbit}\label{coadjoint}

The link between hyperbolic geometry in even dimensions and symplectic geometry was first noticed by Reznikov \cite{Re} using the language of twistor spaces. We briefly give an alternative description via coadjoint orbits, which first appeared in \cite{FP}. We focus purely on the case of four dimensions, which leads to symplectic Calabi--Yau manifolds, although similar considerations apply in higher even dimensions leading to symplectic Fano manifolds.

Let $G$ be a Lie group. It is a standard fact that there is a $G$-invariant symplectic structure on the coadjoint orbits of $G$. We will apply this to a certain coadjoint orbit of $\SO(4,1)$.

The Lie algebra $\so(4,1)$ is  5$\times$5 matrices of the form
\[
\label{matrix}
\left(
\begin{array}{cc}
0 & u^t\\
u & A
\end{array}
\right)
\]
where $u$ is a column vector in $\R^4$ and $A \in \so(4)$. The Killing form is non-degenerate on $\so(4,1)$ and so gives an equivariant isomorphism $\so(4,1) \cong \so(4,1)^*$. We consider the orbit $Z$ of
\[
\xi=\left(
\begin{array}{cc}
0 & 0\\
0 & J_0
\end{array}
\right)
\]
where $J_0\in \so(4)$ is a choice of almost complex structure on $\R^4$ (i.e., $J_0^2=-1$). The subalgebra $\mathfrak h$ of matrices commuting with $\xi$ is those with $u=0$ and $[A,J_0]=0$, i.e., $\mathfrak h = \u(2) \subset \so(4) \subset\so(4,1)$. It follows that the stabilizer of $\xi$ is $\U(2)$ and so $Z \cong \SO(4,1)/\U(2)$.

The key facts we will need about $Z$ are summarized as follows. (See \cite{FP} for details and proofs.)

\begin{enumerate}
\item{\bfseries Calabi--Yau structure.}
As a coadjoint orbit, $Z$ carries an $\SO(4,1)$-invariant symplectic structure.

There is also $\SO(4,1)$-invariant compatible almost complex structure $J_{\text{ES}}$ on $Z$ (named for Eells and Salamon; it is an instance of their almost complex structure on twistor spaces~\cite{eells-salamon}).

The $J_{\text{ES}}$-canonical bundle admits an $\SO(4,1)$-invariant nowhere-vanishing section, making $Z$ a homogenous Calabi--Yau manifold.
\item{\bfseries Twistor fibration.}
The inclusion $\U(2) \subset \SO(4)$ induces an $\SO(4,1)$-equivariant fibration $t \colon Z \to \SO(4,1)/\SO(4) \cong \H^4$.

The fibres $\SO(4)/\U(2)\cong S^2$ of $t$ are $J_{\text{ES}}$-holomorphic and hence symplectic spheres in $Z$.

Fixing an orientation on $\H^4$, the fibre $t^{-1}(p)$ is canonically identified with the set of linear complex structures on $T_p\H^4$ which are orthogonal with respect to the hyperbolic metric and induce the chosen orientation. In other words, $Z$ is the twistor space of $\H^4$.

\item{\bfseries Twistor lifts of 2-planes.}
Given a totally geodesic embedding $\H^2 \to \H^4$ and an orientation on $\H^2$ there is a unique lift $\H^2 \to Z$ which is $J_{\text{ES}}$-antiholomorphic (with respect to the natural complex structure on $\H^2$) and hence symplectic.

The lift of a point $p \in \H^2 \subset \H^4$ is the unique linear complex structure on $T_p\H^4$ in $t^{-1}(p)$ which makes $T_p\H^2 \subset T_p\H^4$ complex linear. (More generally this recipe gives a lift of any immersed oriented surface; these so-called ``twistor lifts'' originated in \cite{eells-salamon}.)

The lift corresponding to the opposite orientation on $\H^2$ is given by composing the first lift with the antipodal map on each fibre of $t$.

Given two distinct oriented 2-planes in $\H^4$, their lifts meet if and only if the  2-planes meet orthogonally in $\H^4$ and their orientations combine to give the chosen one on $\H^4$.
\item{\bfseries Calabi--Yau orbifold quotients.}
By $\SO(4,1)$-equivariance, all the above points apply to hyperbolic manifolds $\H^4/\Gamma$ where $\Gamma \subset \SO(4,1)$. They carry a 2-sphere bundle $Z/\Gamma \to \H^4/\Gamma$, the total space of which is a symplectic Calabi--Yau manifold.

Similarly if $\H^4/\Gamma$ is a hyperbolic orbifold, $Z/\Gamma$ is a symplectic Calabi--Yau orbifold. The orbifold points in $Z/\Gamma$ map to the orbifold points of $\H^4/\Gamma$ under $Z/\Gamma \to \H^4/\Gamma$.

Away from the singular locus, this map is still an $S^2$-fibre bundle. If $D \subset \SO(4)$ is the orbifold group of a point $p \in \H^4/\Gamma$ then the fibre over $p$ is $D\backslash\SO(4)/\U(2)$ (although one should remember that points in the fibre over $p$ may also have orbifold singularities in the transverse directions).

An immersed orientable totally-geodesic sub-orbifold $S \subset \H^4/\Gamma$ has two disjoint lifts to $Z/\Gamma$, distinguished by a choice of orientation on $S$.
\end{enumerate}

\subsection{The 120-cell and hyperbolic orbifolds}\label{120cell}

Here we recall very briefly the construction given in \cite{PP} of the $4$-dimensional hyperbolic orbifolds occurring in Theorem \ref{mainhyper} and, in particular, describe the singularities of these orbifolds.

In \cite{PP} one constructs a compact orientable four-dimensional hyperbolic orbifold that has the {\it telescopic property}, which means that its finite orbi-covers can have arbitrary finitely-presented fundamental groups. To construct this telescopic orbifold one starts with the two-dimensional orbihedron that contains one vertex and two loops $g$ and
$r$ to which fours two-cells are attached along words $g$, $r$, $gr$, $gr^{-1}$;
moreover the interior of each cell contains three orbi-points with stabilizers $\mathbb Z_2$.
Such an orbihedron has telescopic property. The orbihedron is
thickened by attaching to it a collection of right-angled hyperbolic $120$-cells. We recall that a hyperbolic $120$-cell is a regular 4-dimensional hyperbolic Coxeter polytope with $120$ geodesic three-faces, that are dodecahedrons. The term ``right-angled'' means that the neighboring three-faces of the polytope intersect in the angle $\frac{\pi}{2}$. After this thickening one obtains a 4-dimensional hyperbolic orbifold with boundary that can be doubled; this produces the telescopic orbifold.

The construction in Theorem  \ref{mainhyper} is done in such a way that the group $\Gamma$ is a finite index subgroup of the group of isometries of $\mathbb H^4$ generated by reflections
in the faces of the $120$-cell.
The only important information that we need to retain about $\Gamma$ is the action of stabilizers of points in $\mathbb H^4$. The stabilizer sub-groups can be only $\mathbb Z_2$, $\mathbb Z_2^2$, and $\mathbb Z_2^3$. To describe the action of $\mathbb Z_2^3$, we use coordinates $x_1, x_2, x_3, x_4$ in the ball model of $\mathbb{H}^4$. The action is generated by reflections in the coordinate 2-planes, so that each element of $\Z_2^{3}$ acts by an orientation--preserving isometry of the form
\begin{equation}\label{orbiaction}
(x_1,x_2,x_3,x_4)\to (\pm x_1, \pm x_2, \pm x_3, \pm x_4)
\end{equation}
where an even number of signs are reversed. Meanwhile, the actions of $\mathbb Z_2$ and $\mathbb Z_2^2$ are sub-actions of $\mathbb Z_2^{3}$. The simplest example of a hyperbolic orbifold with precisely these singularities is given by doubling a right-angled 120-cell. For more on examples of this kind see \S\ref{simply}.

\subsection{Singularities in the twistor space}\label{twistor_singularities}

As was described in \S\ref{coadjoint}, given an oriented hyperbolic orbifold $\mathbb{H}^4/\Gamma$, its twistor space $Z/\Gamma$ is a symplectic Calabi--Yau orbifold. In the following section we will explain, for those orbifolds arising from Theorem \ref{mainhyper}, how to resolve the singularities to produce smooth symplectic Calabi--Yaus. First we give a local description of these singularities, considering the quotient of the twistor space $Z$ of $\mathbb H^4$ by the action of $\mathbb Z_2^{3}$ given by (\ref{orbiaction}).

For each pair $i\ne j$, let $\Pi_{ij}$ be the geodesic two-plane in $\mathbb H^4$ corresponding to the coordinate two-plane $(x_i, x_j)$. To each point $x \in \Pi_{ij}$ we can associate two orthogonal almost complex structures $\pm J_x, J_x^2=-\rm Id$ on $T_x\mathbb{H}^4$ with respect to which $T_x\Pi_{ij}$ is a complex line. This gives us two lifts of each $\Pi_{ij}$ to the twistor space $Z$ of $\mathbb H^4$. Given $\Pi_{ij}$, each lift intersects precisely one lift of the orthogonal plane $\Pi_{ij}^{\perp}$ in the twistor fiber over the fixed point $p$, the other lifts intersecting in the antipodal point.  Altogether we have $12$ different lifts for different $(i,j)$ and they intersect the twistor fiber over $p$ in a collection of $6$ points arranged as the six vertices of the octahedron. There are no other intersections of the lifts in $Z$.

The action of $\mathbb Z_2^{3}$ on the twistor fiber over $p$ is not faithful, since the central symmetry of $\mathbb H^4$ with respect to $p$ acts trivially. Each of the $6$ points on the fiber where a lift of $\Pi_{ij}$ meets a lift of $\Pi_{ij}^{\perp}$ has stabilizer $\mathbb Z_2^{2}$ equal to the subgroup of $\mathbb Z_2^{3}$ that leaves invariant $\Pi_{ij}$ and $\Pi_{ij}^{\perp}$. Finally, the stabilizer of a point on a lift of $\Pi_{ij}$ not in the fiber over $p$ equals $\mathbb Z_2$.
All this is summarized in the following lemma.

\begin{lemma}
The action of $\mathbb Z_2^{3}$ on the twistor space $Z$ has two types of points with non-trivial stabilizer:
\begin{enumerate}
\item
There are $6$ points with stabilizers $\mathbb Z_2^{2}$, arranged as the vertices of an octahedron on the central twistor fibre. Each pair of opposite vertices forms an orbit of the $\Z_2^{3}$ action.
\item
There are also points with stabilizer $\mathbb Z_2$. These are the union of the central twistor fibre---minus the six points---together with the twelve surfaces which are the lifts of the six totally-geodesic coordinate planes $\Pi_{ij} \subset \mathbb{H}^4$.
\end{enumerate}
\end{lemma}

We also need the following description of the action of the stabilizer groups near these points.

\begin{lemma}\label{local_model_orbipoints}
We have the following local models for the singularities of $Z/\Z_2^3$.
\begin{enumerate}
\item
Let $q \in Z$ be a point with stabilizer $\Z_2$. There are local complex coordinates $(z_1,z_2,z_3)$ centred at $q$, in which the symplectic structure is standard and in which the $\Z_2$ action is generated by the transformation
\begin{equation}\label{z2act}
(z_1,z_2,z_3) \mapsto (-z_1, -z_2, z_3).
\end{equation}
\item
Let $p\in Z$ be a point with stabilizer $\Z_2^2$. There are local complex coordinates $(z_1, z_2, z_3)$ centred at $p$, in which the symplectic structure is standard and in which each element of the stabilizer of $p$ acts by a transformation of the form
\begin{equation}\label{z2z2act}
(z_1,z_2,z_3)\mapsto (\pm z_1,\pm z_2, \pm z_3),
\end{equation}
where an even number of signs are reversed.
\end{enumerate}
\end{lemma}

This is a direct consequence of the following theorem \cite{GS}:

\begin{theorem}[Equivariant Darboux Lemma] Let $(M, \omega)$ be a
$2n$-dimensional symplectic manifold equipped with a symplectic action
of a compact Lie group $G$, and let $q$ be a fixed point. Then there exists a $G$-invariant chart $(U, x_1,...,x_n, y_1,...,y_n)$ centered at $q$ and $G$-equivariant with respect to a linear action of $G$ on $\mathbb R^{2n}$ such that
\[
\omega|_U =\sum_{k=1}^n dx_k \wedge dy_k.
\]
\end{theorem}

\section{Resolving the twistor space}\label{resolution}

To resolve the singularities of $Z/\Gamma$, we will throw away a neighbourhood of the singular locus, leaving a manifold with corners. The boundary of this manifold with corners will be foliated by tori of different dimensions; collapsing these tori will give a smooth manifold with a symplectic structure---the blow up of $Z/\Gamma$ along the singular locus. The key to finding the manifold with corners is the fact that the singularities are abelian, which means they can locally be described using toric geometry and, moreover, that these local pictures are mutually compatible in a precise sense.

We begin in \S\ref{toric_model} by explaining the local toric models for the singularities and their resolutions. In \S\ref{cutting_thm} we consider what it means in general for local toric models to be ``mutually compatible''. We explain how, when this compatibility occurs, collapsing tori on the boundary of a certain manifold with corners yields a symplectic manifold. This is a ``semi-local'' version of the standard Delzant construction, which gives a symplectic toric manifold from a Delzant polytope \cite{delzant}. Finally  in \S\ref{doing_the_blowup} we apply this to the orbifolds $Z/\Gamma$ arising from Theorem \ref{mainhyper} to blow up the singular locus and produce the sought-after resolution.

\subsection{Local toric models for the resolution}\label{toric_model}

We assume the reader is familiar with toric geometry and, in particular, the symplectic approach based on moment polytopes and the work of Delzant \cite{delzant}.

\subsubsection*{Model resolution near a $\Z_2$-point}

Lemma \ref{local_model_orbipoints} gives that the singularity at a $\Z_2$-orbifold point is modeled on the quotient of $\C^3$ by $\Z_2$ acting as $(z_1,z_2,z_3) \mapsto (-z_1,-z_2,z_3)$. In other words, the model singularity is $A_1 \times \C$ and this has the crepant resolution $\O(-2) \times \C$. We can describe this from the symplectic toric point of view as follows. There is a $T^2$-action on $A_1$ coming from the $T^2 \subset \SU(2)$ action on $\C^2$; together with the $S^1$-action on $\C$ this gives a $T^3$-action on $A_1 \times \C$. The moment polytope of this action is given by the convex hull in $\R^3$ of the rays through the points $(1,0,0)$, $(1,2,0)$ and $(0,0,1)$ or,
equivalently, the inequalities
\[
2x - y \geq 0,
\quad
y \geq 0,
\quad
z\geq 0.
\]
The ray $x=0=y$ corresponds to the singular locus. In the resolution we  remove this ray by adding the inequality $x \geq 1$. One readily checks the new polytope is Delzant and corresponds to the blow-up of the singular locus in $\C^3/\Z_2$. (Strictly speaking, of course, the polytope gives the resolution together with a particular choice of symplectic structure corresponding to the chosen size of the exceptional locus; we habitually omit to mention this choice of scale, calling the result ``the'' blow-up.)

\subsubsection*{Model resolution near a $\Z_2\oplus \Z_2$-point}

It follows from Lemma \ref{local_model_orbipoints} that the singularity at a $\Z_2\oplus \Z_2$ point is locally given by the quotient of $\C^3$, with its Euclidean symplectic structure, by the action (\ref{z2z2act}) of $\Z_2 \oplus \Z_2$. This action commutes with the standard $T^3$-action on $\C^3$ and so the result is again a toric orbifold. It has moment-polytope $P$ which is the convex hull of the rays through $(1,1,0)$, $(1,0,1)$ and $(0,1,1)$ or, equivalently, the region of $\R^3$ given by the inequalities
\[
x + y - z \geq 0, \quad x - y + z \geq 0, \quad -x + y + z \geq 0.
\]
The origin gives the point with orbifold group $\Z_2\oplus \Z_2$ whilst the three rays correspond to those points with orbifold group $\Z_2$.

To resolve we blow up the singular locus. Symplectically this corresponds to cutting the polytope along hyperplanes in order to make it Delzant. There are different choices possible here; the resolution we focus on is the most symmetric, where we impose the additional inequalities $x\geq1$, $y\geq1$ and $z\geq1$. This gives a polytope $R$ with four vertices, at $(1,1,1)$, $(1,1,2)$, $(1,2,1)$ and $(2,1,1)$. The normals at $(1,1,1)$ are just the standard coordinate vectors whilst at, say, $(1,1,2)$ the normals are $(1,0,0)$, $(0,1,0)$ and $(1, 1, -1)$; at each vertex then the normals form a basis for the lattice $\Z^3$ and so $R$ is Delzant.

\begin{center}

\hspace{1.5\baselineskip}

\includegraphics[width=8cm]{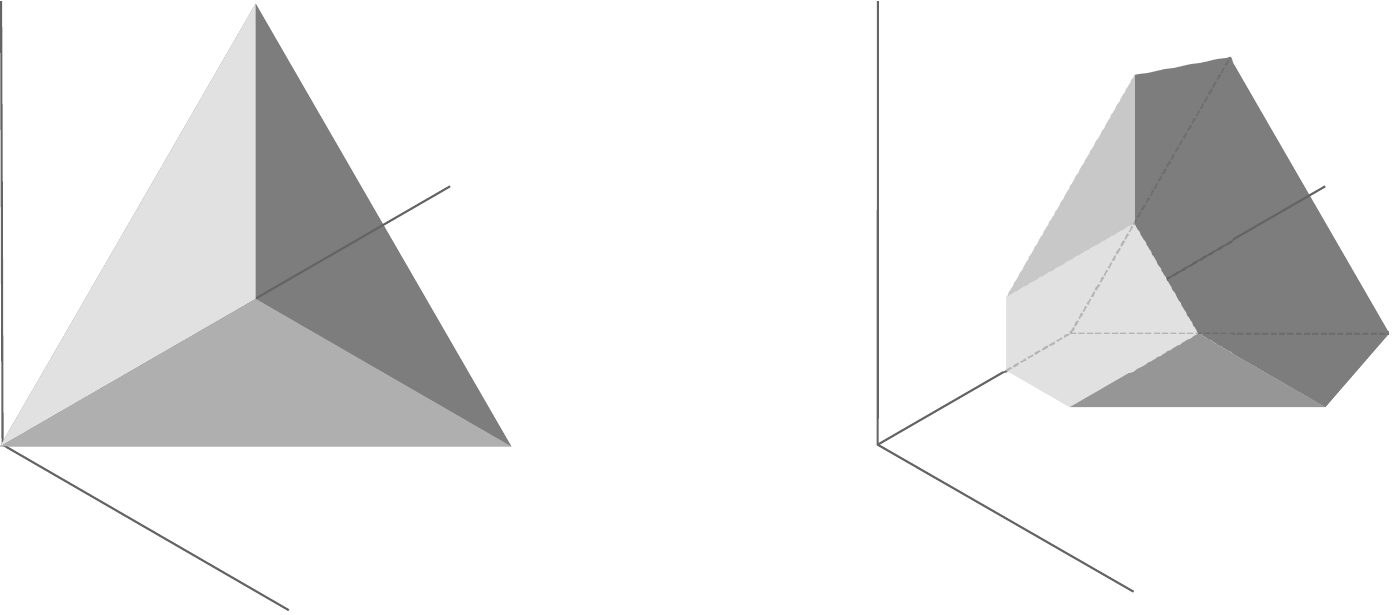}

{\itshape The left-hand polytope $P$ is the singularity $\C^3/\Z_2^2$, the right-hand polytope $R$ is the resolution given by blowing up the singular locus.}
\end{center}

Each ray in $P$ has been replaced by a 2-face in $R$, giving three isomorphic toric divisors $E_1,E_2, E_3$ in the resolution, each isomorphic to the blow-up of $\C\P^1 \times \C$. (For example, the 2-face in the plane $x=1$ has polytope given by the inequalities, $y,z \geq 1$ and $-1\leq y- z \leq 1$.) Each divisor $E_i$ contains two $-1$-curves, of equal area, and each intersection $E_i \cap E_j$ occurs along one of these curves. There are three such curves $C_1, C_2, C_3$ in all, meeting in a single point corresponding to the vertex $(1,1,1)$. Each curve $C_i$ has normal bundle $\O(-1)\oplus \O(-1)$.

Both resolutions described here are produced by blowing up the singular locus. In these toric model situations, the symplectic structure on the blow-up is clear. In the non-toric case, however, there is no result of sufficient generality available in the literature which immediately gives the existence of a symplectic blow-up of $Z/\Gamma$ along its singular locus (at least to the best of our knowledge). To proceed then, we first describe a semi-local version of the Delzant construction which will enable us to symplectically blow up orbifold singularities which are ``semi-locally toric'' in nature. We then apply this machinery to blow-up $Z/\Gamma$.

\subsection{Symplectic cuts and semi-local torus actions}
\label{cutting_thm}

To construct the blow-up of $Z/\Gamma$ we will use the language of symplectic cutting, introduced by Lerman \cite{lerman}; see also the article \cite{lerman-woodward} which describes the version of cutting involving several commuting Hamiltonians which we will use. We briefly recall the definition here.

\begin{definition}\label{mulitcut}
Let $M$ be a symplectic manifold with a Hamiltonian $T^k$-action and let $h_1, \ldots , h_k$ be Hamiltonian functions corresponding to an integral basis of $\Lie(T^k)$. We consider the product $T^k$-action on $M \times \C^k$ which has moment map $\mu \colon M \times \C^k \to \R^k$
\[
\mu(p,z) = \left(h_1(p) - |z_1|^2, \ldots , h_k(p) - |z_k|^2 \right).
\]
Suppose that $c = (c_1,\ldots, c_k)$ is a regular value of $h = (h_1, \ldots, h_k) $ and that the copy of $S^1 \subset T^k$ generated by $h_j$ acts freely on the level set $h_j=c_j$. Then the \emph{symplectic cut of $M$ by $h$ at level $c$} is the symplectic manifold $M_{h,c} = \mu^{-1}(c)/T^k$, i.e., the reduction of $M \times \C^k$ at $\mu = c$. (The fact that $c$ is a regular value of $\mu$ and that $T^k$ acts freely on $\mu^{-1}(c)$ are ensured by the analogous conditions for the original torus action.)
\end{definition}

The case $k=1$ is perhaps more familiar. There, the open set $\{ h > c\}$ symplectically embeds as a dense open subset of $M_{h,c}$; topologically one can regard $M_{h,c}$ as the quotient obtained from the manifold-with-boundary $\{h \geq c\}$ by collapsing the $S^1$-orbits in the boundary. When there are more Hamiltonians we have a similar picture. Again $\{ h_i > c_i : i = 1, \ldots, k\}$ is a symplectically embedded dense open set $M'_{h,c} \subset M_{h,c}$. The boundary of $\overline{M}'_{h,c}$, made up of the hypersurfaces $h_i = c_i$, is stratified with the codimension $m$ stratum coming from the intersection of~$m$ bounding hypersurfaces. Each component of the codimension $m$ stratum carries a $T^m$-action, generated by the $m$ Hamiltonians whose level-sets meet there. To visualize the cut manifold $M_{h,c}$ topologically, collapse the orbits of these torus actions on the various strata of the boundary of $\overline{M}_{h,c}'$.

Our use of symplectic cutting hinges on the following simple lemma.

\begin{lemma}\label{compatibility}
Let $M$ be a symplectic manifold with a Hamiltonian action of $T^k$ generated by commuting Hamiltonians $h_1, \ldots, h_k$, let $c
\in \R^k$ be a regular value of $h= (h_1, \ldots, h_k)$ and suppose that $T^k$ acts freely on the level set $h=c$. Suppose, moreover, that for some $r<k$ the first $r$ Hamiltonians satisfy $h_i > c_i$ at all points of $M$. Then the cut of $M$ by $h$ at $c$ is symplectomorphic to the cut of $M$ by $(h_{r+1}, \ldots, h_k)$ at $(c_{r+1}, \ldots, c_k)$.
\end{lemma}

\begin{proof}
Write $X$ for the cut of $M$ by $h$ at $c$ and $Y$ for the cut of $M$ by $(h_{r+1}, \ldots , h_k)$ at $(c_{r+1}, \ldots, c_k)$. By definition, $Y$ is the quotient of the following subset of $M \times \C^{k-r}$ by $T^{k-r}$:
\[
\hat{Y} = \{ (p, w_{r+1}, \ldots, w_k)  : h_i(p) = |w_i|^2 + c_i,\ i=r+1,\ldots, k \}
\]
We write $\hat{X} \subset M \times \C^k$ for the analogous set whose quotient is $X$. There is a natural $T^{k-r}$-equivariant map $\phi \colon \hat{Y} \to \hat{X}$ given by
\[
\phi(p, w_{r+1}, \ldots, w_k)
=
\left(
p, \sqrt{h_1(p) - c_1}, \ldots, \sqrt{h_r(p) - c_r}, w_{r+1}, \ldots, w_k\right)
\]
The map $\phi$ is well-defined and smooth precisely because of the assumption that $h_i > c_i$ for all $i=1, \ldots , r$. Moreover, if $(p,w) \in \hat{X}$ then there is a unique element of $T^r$ which moves $(p,w)$ to lie in the image of $\phi$.  In other words, $\phi$ embeds $\hat{Y}$ as a slice in $\hat{X}$ for the $T^r$-action. It follows that $\phi$ induces a diffeomorphism between the quotients $\hat{X}/T^k =X$ and $\hat{Y}/T^{r-k} = Y$.

Both $\hat{X}$ and $\hat{Y}$ carry closed 2-forms $\omega_{\hat{X}}$ and $\omega_{\hat{Y}}$, by restriction from the ambient spaces $M \times \C^k$ and $M \times \C^{k-r}$. These 2-forms have kernel exactly along the torus orbits and descend to define the symplectic structures on the quotients $X$ and $Y$. To see then that $\phi$ induces a symplectomorphism between $X$ and $Y$, it suffices to observe that $\phi^*(\omega_{\hat{X}} ) = \omega_{\hat{Y}}$.
\end{proof}

This lemma enables us to extend cutting to manifolds with local torus actions, where the dimensions of the tori may vary. To explain how this goes we begin by rephrasing a familiar example in this language, namely that of recovering a genuine toric manifold from a Delzant polytope $P \subset \R^n$ \cite{delzant}. We begin with $T^n \times \R^n$ with the obvious symplectic form which makes the projection onto either factor a Lagrangian fibration. Let $U \subset \R^n$ be an open set which contains a single vertex $v$ of $P$. There are $n$ linear functions defined up to scale on $U$ which vanish on the $n$ facets of $P$ meeting at $v$ and which are positive on the interior of $P$. The pull-back of these linear functions to $T^n \times U$ define independent Hamiltonian vector fields parallel to the $T^n$-fibres and, because the vertex is Delzant, we can scale the Hamiltonians so that they are an integral basis for the obvious $T^n$ action on $T^n \times \R^n$. Next suppose that $U \subset \R^n$ is an open set which meets a single one-dimensional edge. Then on $T^n \times U$ we have $n-1$-Hamiltonians generating a $T^{n-1}$-action, corresponding to the $n-1$ facets of $P$ which meet along this edge. Similarly if $U \subset \R^n$ is an open set meeting a single dimension $k$ face then we have $n-k$ Hamiltonians coming from the $n-k$ facets meeting there.

Now we cover the whole of $P$ in such sets and take the cuts of each $T^n\times U$ by the corresponding Hamiltonians at level $0$. In each case the cut of $T^n \times U$ yields a symplectic manifold projecting to $U \cap P$. When two such open sets $U, V$ intersect, the Hamiltonians of one, $U$ say, are necessarily a subset of those of $V$. Moreover, the additional Hamiltonians coming from $V$ are strictly positive on $U \cap V$. So  Lemma \ref{compatibility} allows us to glue the cuts of $T^n \times U$ and $T^n \times V$ together to obtain a symplectic manifold projecting to $(U \cup V) \cap P$. Doing this with every open set in the cover of $P$ and gluing according to the intersections we obtain a symplectic manifold $X$ together with a map $X \to P$. Note that for each open set $U$ the torus action on $T^n \times U$ defines a torus action on the cut manifold, whose moment map is the projection to $U \cap P$. These torus actions fit together under the gluings to give a torus action on $X$ with moment map the projection to $P$. The manifold $X$ is, of course, the toric symplectic manifold associated to the polytope by Delzant.

This description is not so far from Delzant's original construction, which involved a global symplectic reduction. But by placing the emphasis on the local behvaiour and the gluings we are able to generalize the construction to deal with things which are only ``semi-locally'' toric. This is the motivation behind the following definition.

\begin{definition}\label{local_cutting_data}
Let $M$ be a symplectic manifold.  \emph{Semilocal cutting data on $M$} is the following:
\begin{enumerate}
\item
An open cover $\{U_\alpha\}$ of $M$.
\item
For each $\alpha$, a collection of Hamiltonians $h_{\alpha,1}, \ldots , h_{\alpha, k_\alpha} \colon U_\alpha \to \R$ generating a $T^{k_\alpha}$-action on $U_\alpha$ (where we allow $k_\alpha = 0$). We require the Hamiltonians to form an integral basis for the torus action.
\item
For each $\alpha$ an element $c_\alpha \in \R^{k_\alpha}$ which is a regular value of $h_\alpha$. We assume, moreover, that the action of $S^1 \subset T^{k_\alpha}$ generated by $h_{\alpha, j}$ acts freely on $h_{\alpha, j}^{-1}(c_{\alpha, j})$.
\end{enumerate}
The data must satisfy the following compatibility conditions:
\begin{enumerate}
\item
If $U_\alpha \cap U_\beta \neq \varnothing$ then it is both $T^{k_\alpha}$- and $T^{k_\beta}$-invariant.
\item
If $U_\alpha \cap U_\beta \neq \varnothing$ and $k_\alpha \geq k_\beta$ then there is a way to reorder the Hamiltionians on $U_\alpha$ and $U_\beta$ so that on the intersection $h_{\beta,i} = h_{\alpha,i}$  for $i = 1,\ldots , k_\beta$.
\item
If $U_\alpha \cap U_\beta \neq \varnothing$ then, with the above reordering, we require that $c_{\beta,i} = c_{\alpha,i}$ for $i=1, \ldots , k_\beta$.
\item
If $U_\alpha \cap U_\beta \neq \varnothing$ then, with the above reordering, we ask that $h_{\alpha,i} > c_{\alpha,i}$ for $i = k_\beta + 1, \ldots , k_\alpha$.
\end{enumerate}
\end{definition}

Given symplectic cutting data, the subsets
\[
\{
(h_{\alpha, 1}, \ldots, h_{\alpha, k_\alpha})
\geq
(c_{\alpha,1}, \ldots , c_{\alpha, k_\alpha})
\}
\subset U_\alpha
\]
fit together to give a closed subset $\overline{M}' \subset M$ whose boundary is stratified by torus orbits, just as in the case of cutting with a globally defined collection of Hamiltonians. Again, just as in the global situation, collapsing these torus orbits produces a topological manifold $\widetilde{M}$. Now each of the cut pieces carries a natural symplectic structure and Lemma \ref{compatibility} tells us that these match up on the overlaps. This proves:

\begin{theorem}[Semi-local Delzant construction]
\label{local_torus_cuts}
Let $M$ be a symplectic manifold, $\{(U_\alpha, h_\alpha, c_\alpha)\}$ semi-local symplectic cutting data on $M$ and $\widetilde{M}$ the topological manifold produced by collapsing the torus orbits on the boundary of $\overline{M}' \subset M$. Then there is a symplectic structure on $\widetilde{M}$ which agrees on the image of each $U_\alpha \cap \overline{M}'$ with the cut of $U_\alpha$ by $h_\alpha$ at $c_\alpha$.
\end{theorem}

\subsection{Blowing up the singular locus in the twistor space}\label{doing_the_blowup}

We return now to the case of a symlpectic orbifold $Z/\Gamma$ arising as the twistor space of a hyperbolic orbifold from Theorem~\ref{mainhyper}.  We will find local Hamiltonians on $Z/\Gamma$ in order that Theorem~\ref{local_torus_cuts} can be applied to cut away a complement of the orbifold points. We will arrange things so that, for example, near a point with orbifold group $\Z_2 \oplus \Z_2$ this symplectic cutting carries out the blow-up resolution described in the local toric model of \S\ref{toric_model}.

We begin by fixing disjoint neighborhoods $U_p$ of each $\Z_2 \oplus \Z_2$-point $p$ where the singularity is modeled on the quotient of $\C^3$ by the action (\ref{z2z2act}) of $\Z_2 \oplus \Z_2$. Here we certainly have a Hamiltonian $T^3$-action, which is described in \S\ref{toric_model}. Writing $x,y,z$ for the coordinate Hamiltonians, the blow-up of the singular locus in $U_p$ is given by the symplectic cut by of $U_p$ by $(x,y,z)$ at $(\epsilon, \epsilon, \epsilon)$ for $\epsilon >0$ sufficiently small.

Next we consider an irreducible component $\Sigma$ of the singular locus which has a certain number $p_1, \ldots , p_m$ of $\Z_2\oplus\Z_2$-points, with the remainder of the points being $\Z_2$-points. In each $U_{p_i}$ we have a distinguished Hamiltonian $S^1$-action which fixes $\Sigma$ and rotates the normal directions with both weights equal to $1$. If $\Sigma$ corresponds to the ray $\{x=0,\ y = z\}$  in the local toric model, then the Hamiltonian of this distinguished $S^1$-action is $x$. We want to extend the circle actions on each $U_{p_i}$ to a Hamiltonian $S^1$-action defined on a neighborhood of the whole of $\Sigma$.  To this end we begin with the following version of Weinstein's symplectic neighborhood theorem.

\begin{theorem}[Weinstein \cite{weinstein}]
\label{relative_local_neighbourhood}
Let $(M, \omega)$ and $(M', \omega')$ be two symplectic orbifolds and let $N\subset M$ and $N' \subset M'$ be symplectic sub-orbifolds with normal orbifold-bundles $E$ and $E'$ respectively.
\begin{enumerate}
\item
Let $D \colon E \to E'$ be an isomorphism of symplectic orbifold-bundles covering a symplectomorphism $\phi_0 \colon N \to N'$ downstairs. Then there exist neighborhoods $W$ and $W'$ of $N$ and $N'$ and a symplectomoprhism $\phi \colon W\to W'$ extending $\phi_0$ such that $D\phi = D$ on $E$.
\item
Suppose in addition that there are points $p_1, \ldots , p_m \in N$ and locally defined symplectomorphisms $\psi_1, \ldots , \psi_m$ with $\psi_i$ defined near $p_i$, taking values in $M'$ and with $\psi_i|_N = \phi_0$. Suppose, moreover, that on the normal bundle of $N$, we have $D = D\psi_j$. Then we can arrange that for each $i$, there is a neighborhood of $p_i$ (possibly smaller than the domain of $\psi_i$) on which $\phi = \psi_i$.
\end{enumerate}
\end{theorem}

Part 1 is precisely Weinstein's Theorem as it is usually stated. (In fact it is more normally given for manifolds, but the proof applies equally to orbifolds.) Simply paying attention during this proof to what happens near the $p_j$ shows that part 2 holds as well.

\begin{lemma}\label{extending_circle_actions}
Let $\Sigma \subset Z/\Gamma$ be an irreducible component of the singular locus with $p_1, \ldots, p_m \in \Sigma$ the $\Z_2\oplus \Z_2$-points. Then there are neighbourhoods $V_{p_i} \subset U_{p_i}$ of each $p_i$ and a Hamiltonian $S^1$-action on a neighborhood of $\Sigma$ which fixes $\Sigma$ and such that over $V_{p_i}$ the action agrees with the distinguished action defined by the local toric model at $p_i$.
\end{lemma}

\begin{proof}
First we note that, at the cost of shrinking the $U_{p_i}$ slightly, the $S^1$-action defined on each $U_{p_i}$ extends to a smooth action (not necessarily preserving the symplectic form) on a neighborhood $A$ of $\Sigma$, and which fixes $\Sigma$ pointwise.

Now we average the symplectic form $\omega$ under this $S^1$-action to obtain a new symplectic form $\omega'$ on $A$ which is $S^1$-invariant and agrees with $\omega$ on $\Sigma \cup U_{p_1} \cup \cdots \cup U_{p_m}$.
The $S^1$-action is Hamiltonian with respect to $\omega'$, since a neighborhood of $\Sigma$
can be contracted to $\Sigma$ and $\Sigma$ is fixed by $S^1$.

Finally, we use the preceding neighborhood theorem to deduce that there are neighborhoods $W,W' \subset A$ of $\Sigma$, open sets $V_{p_i}\subset U_{p_i} \cap W$ with $p_i \in V_{p_i}$ and a diffeomorphism $\phi \colon W \to W'$ such that $\phi^*\omega' = \omega$ and $\phi$ is the identity on both $\Sigma$ and on $V_{p_i}\cap W$. Pulling back the circle action from $W'$ to $W$ via $\phi$ we obtain a Hamiltonian circle action for $\omega$ which fixes $\Sigma$ pointwise and agrees with the distinguished ones on $V_{p_i}$.
\end{proof}

With this result behind us we can now blow up the singular locus

\begin{theorem}
Let $Z/\Gamma$ be an orbifold twistor space arising from Theorem \ref{mainhyper} or, more generally, associated to a hyperbolic orbifold $\H^4/\Gamma$ with all singularities modelled on (\ref{orbiaction}). Then there exists a smooth symplectic blow-up $\widetilde{Z}$ of $Z/\Gamma$ along its singular locus.
\end{theorem}

\begin{proof}
We have the following local Hamiltonian torus actions: we first fix a choice of Hamiltonian $T^3$-action near each $\Z_2 \oplus \Z_2$ point by choosing a local model there as in (\ref{z2z2act}); then Lemma \ref{extending_circle_actions} gives us a Hamiltonian $S^1$-action near each $\Z_2$-point. Normalising the Hamiltonians so that each vanishes on its fixed locus ensures that the compatibility requirements of Definition \ref{local_cutting_data} are satisfied, meaning we can cut by all Hamiltonians at level $\epsilon>0$ for some sufficiently small choice of $\epsilon$. In the case of a single Hamiltonian, the cut is symplectomorphic to a neighbourhood  of $\C\P^1$ in the resolution $\C \times \O(-2)$ of $\C \times A_1$. In the case of three Hamiltonians, the cut is symplectomorphic to a neighbourhood of the three curves $C_1, C_2, C_3$ in the toric model of \S\ref{toric_model}. It follows that the cut is a smooth symplectic manifold, which we denote by $\widetilde{Z}$.

We are justified in calling $\widetilde{Z}$ the blow-up of $Z/\Gamma$ since near each point of the singular locus we have a local model in which the cutting carries out the blow up. This means that $\widetilde{Z}$ has the diffeomorphism type of the smooth blow-up. Moreover, writing $E \subset \widetilde{Z}$ for the exceptional locus (those points corresponding to the collapsed boundary torus orbits when making the cut) it follows from the construction that $\widetilde{Z} \setminus E$ embeds symplectically as an open set $Z/\Gamma$, disjoint from the singular locus.
\end{proof}

\section{The topology of the resolution}\label{topology}

\subsection{The first Chern class}
In this section we prove that the resolution $\widetilde{Z}$ of $Z/\Gamma$ constructed in the previous section has vanishing first Chern class.

We will need the following standard topological lemma.

\begin{lemma} \label{chern1} Let $M$ be a compact oriented manifold with a line bundle $L$. Suppose that $M_1,...,M_n$ is a collection of codimension-two oriented submanifolds of $M$, and $U$ is neighborhood of $\cup_i M_i$ for which $\cup_i M_i$ is a deformation retract. Suppose that there is a non-vanishing section $s$ of $L$ defined on $M\setminus U$. Then $c_1(L)$ is Poincare dual to $\sum_ia_i[M_i]$ where $[M_i]$ are the fundamental classes of the $M_i$ and $a_i\in \mathbb Z$.
\end{lemma}

\begin{proposition}\label{c1_vanishes}
Let $\widetilde{Z}$ be the blow-up of the singular locus of an orbifold twistor space $Z/\Gamma$ arising from Theorem \ref{mainhyper}. Then $\widetilde{Z}$ is a symplectic Calabi--Yau manifold.
\end{proposition}

\begin{proof}
Denote by $E=\cup_i E_i$ the union of all irreducible exceptional divisors $E_i$ of $\widetilde{Z}$ and let $U$ be an open neighborhood of $E$ such that $E$ is a deformation retract of $U$. Recall that $\widetilde Z\setminus U$ can be symplectically identified with an open subset of $Z$ containing only its smooth points. Hence the Eells--Salamon almost complex structure on $Z$ induces an almost
complex structure on $\widetilde Z\setminus U$. We extend this structure to the whole of $\widetilde Z$ and denote the corresponding canonical bundle by $L$. As was explained in Section \ref{coadjoint} the bundle $L$ has a non-vanishing section $s$ over $\widetilde {Z}\setminus U$, and so, by Lemma \ref{chern1}, $c_1(L)$ is Poincare dual to $\sum a_i [E_i]$. To finish the proof we will show that $a_i=0$ for all $i$.

This follows from the local description of the resolution given in Section~\ref{toric_model}. Near each point $p\in E_i$ that does not lie on $E_j$ for $j\ne i$, the resolution is locally modeled on the resolution $\O(-2) \times \C$ of $A_1\times \C$. Here $A_1 = \C^2/\Z_2$, with resolution the total space of $\O(-2) \to \C$. The exceptional locus $E_i$ corresponds to $\C\P^1 \times \C$, where $\C\P^1 \subset \O(-2)$ is the zero section. It follows that there is a $2$-sphere $S^2 \subset \widetilde{Z}$ such that $p\in S^2$, $\langle S^2, E_i \rangle=-2$; moreover $S^2$ does not intersect the divisors $E_j$ for $j\ne i$ and, finally, the restriction of the canonical bundle to $S^2$ is trivial.
\end{proof}

\subsection{The fundamental group}

In this section we prove that $\pi_1(\widetilde{Z})$ is isomorphic to $\pi_1(\H^4/\Gamma)$. To do this we will use the following result. Before stating it, we recall that a space is said to be $LC^n$ if every point has a neighborhood which is $n$-connected (i.e., all homotopy groups up to and including $\pi_n$ are trivial). In particular, all $CW$-complexes are $LC^n$ for every $n$, which is the only situation we will consider. Moreover, we need only the $n=2$ case of the following result.

\begin{theorem}[Smale, \cite{smale}]\label{fundamental_group_theorem}
If $f \colon X \to Y$ is a proper, surjective map of connected, locally compact, separable metric spaces,  $X$ is $LC^n$, and each fibre is $LC^{n-1}$  and $(n-1)$-connected, then the induced homomorphism $\pi_j(X) \to \pi_j(Y)$  is an isomorphism for $j\leq n-1$  and surjective for $j=n$.
\end{theorem}

\begin{proposition}\label{fundamental-group-prop}
Let $\widetilde{Z}$ be the blow-up of the singular locus of an orbifold twistor space $Z/\Gamma$ arising from Theorem \ref{mainhyper}. Then $\pi_1(\widetilde{Z}) = \pi_1(\H^4/\Gamma)$.
\end{proposition}

\begin{proof}
There is a continuous surjective map $f \colon \widetilde{Z} \to Z/\Gamma$. This is because topologically $\widetilde{Z}$ is the blow-up of the singular locus in $Z/\Gamma$. The fibre of $f$ away from the singular locus is a single point. The fibre of $f$ at a $\Z_2$-orbifold point is a 2-sphere, since here the blow-up is modeled on the resolution $\C \times \O(-2) \to \C \times A_1$. At a $\Z_2 \oplus \Z_2$-orbifold point the fibre of $f$ is three 2-spheres meeting at a single point, because here the resolution is modeled on the toric picture described in \S\ref{toric_model}. In all cases the fibre of $f$ is simply connected and so, by Theorem \ref{fundamental_group_theorem}, $f$ induces an isomorphism $\pi_1(\widetilde{Z}) \to \pi_1(Z/\Gamma)$.

Next we consider the twistor projection $t \colon Z/\Gamma \to \H^4/\Gamma$. The fibre of $t$ over a smooth point is a 2-sphere. Over an orbifold point we have an orbifold 2-sphere. Either way the fibre is simply connected and so, by Theorem \ref{fundamental_group_theorem}, $t$ induces an isomorphism $\pi_1(Z/\Gamma) \to \pi_1(\H^4/\Gamma)$.
\end{proof}

\subsection{The second and third Betti numbers}
We now compute the Betti numbers of $\widetilde{Z}$ and prove Theorem \ref{main}. We will need to use a slightly strengthened version of Theorem \ref{mainhyper} (\cite{PP}, Lemma 5.2).

 \begin{theorem}\label{b3} Fix positive integers $c$ and $g$ and a finitely presented group $G$. Then there exists a compact oriented hyperbolic orbifold $\H^4/\Gamma$, with local orbifold groups $\mathbb Z_2^k$, $k=1,2,3$, with actions modelled by (\ref{orbiaction}),  having $\pi_1(\H^4/\Gamma)\cong G$, and such that the irreducible components of the singular locus in the corresponding twistor space $Z/\Gamma$ include at least $c$ genus-$g$ curves of $\mathbb Z_2$-singularities.
\end{theorem}

\begin{proposition}\label{b2b3} Let $\H^4/\Gamma$ be a hyperbolic orbifold arising from Theorem \ref{mainhyper} or, more generally, one in which all local orbifold groups are $\Z_2^k$, $k=1,2,3$, with actions modelled on (\ref{orbiaction}). Write $Z/\Gamma$ for the orbifold twistor space. Let $n$ be the number of irreducible components of the singular locus of $Z/\Gamma$ and let $m$ the sum of their genuses. Then for the blow-up $\widetilde{Z}$ of $Z/\Gamma$ along the singular locus we have:  $b_2(\widetilde{Z})-b_2(Z/\Gamma)=n$ and
$b_3(\widetilde{Z})-b_3(Z/\Gamma)=2m$.

\end{proposition}

\begin{proof} In what follows, homology groups are taken with real coefficients. We begin with~$b_2$. Let $E_1,...,E_n$ be the exceptional divisors of the blow up $f:\widetilde{Z}\to Z/\Gamma$. Poincaré duality applies to compact oriented orbifolds so to prove $b_2(\widetilde{Z})-b_2(Z/\Gamma)=n$ it suffices to show that the fundamental classes $[E_1],...,[E_n]$ are independent in $H_4(\widetilde {Z})$ and that $f_*$ restricts to an isomorphism on $\langle[E_1],...,[E_n]\rangle^{\perp}$.

The classes $[E_i]$ are independent since each contains a $2$-sphere $S_i$ such that $E_i\cdot S_i=-2$, while $E_j\cdot S_i=0$ for $j\ne i$ (as was used in the proof of Proposition \ref{c1_vanishes} that $c_1(\widetilde{Z}) = 0$). We next prove that $f_*$ is injective on $\langle [E_i] \rangle^\perp$, i.e., that if $h\in H_2(\widetilde{Z})$ is orthogonal to all $[E_i]$ and if $f_*(h)=0$ then $h=0$. To see this, represent $h$ by a cycle $\Sigma$ disjoint from the exceptional locus $\bigcup E_i$; we will show that $f(\Sigma)$ is the boundary of a $3$-chain $Y$ in $Z/\Gamma$ that has no intersection with the singular locus of $Z/\Gamma$, the inverse image of $Y$ thus bounds $\Sigma$ showing $h=0$. To find $Y$, we first lift to a manifold, in order to use general position arguments. We take a finite index torsion free subgroup $\Gamma' \to \Gamma$, then $Z/\Gamma' \to Z/\Gamma$ is a finite cover by a manifold. The preimages of the orbifold-points form a $2$-cycle $C$ in $Z/\Gamma'$. Next, choose any $3$-chain bounding $f(\Sigma)$ and consider the inverse image, a $3$-chain $Y'$ in $Z/\Gamma'$. By a general position argument we can, if necessary, find another $3$-chain in $Z/\Gamma'$, again with boundary the inverse image of $f(\Sigma)$, but which does not  intersect $C$. Projecting back to $Z/\Gamma$ gives a $3$-chain $Y''$ bounding some multiple of $f(\Sigma)$. Dividing by this multiple gives the $3$-chain $Y$ disjoint from the singular locus with boundary~$f(\Sigma)$.

Next we must check that the restriction of $f_*$ to $\langle [E_1], \ldots, [E_n] \rangle^\perp$ is surjective. To see this, note that $f$ has degree 1 so, by Poincaré duality, the map $f^* \colon H^4(Z/\Gamma) \to H^4(\widetilde{Z})$ is injective; moreover, its image lies in the annihilator of $\langle [E_1], \ldots, [E_n]\rangle$. This is the dual of what we wanted to show.

We next calculate $b_3$, for which we only need now to prove that the Euler characteristics of $\widetilde{Z}$ and $Z/\Gamma$ satisfy $\chi(\widetilde{Z}) = \chi(Z/\Gamma) + 2(n - m)$ (indeed $b_1$ and $b_5$ do not change whilst  $b_2$ and $b_4$ increase by $n$). The blow up $f: \widetilde{Z}\to Z/\Gamma$ can be represented topologically as sequence of blow ups of singular curves of $Z/\Gamma$. At each step we replace a curve of genus $g$ by a divisor that projects to the curve with $S^2$-fibers. This means that after each blow up the Euler characteristic increases by~$2(1-g)$.
\end{proof}

\subsection{\label{simply}Simply connected Calabi--Yau manifolds}

In this section we prove Theorem \ref{non-Kahler_simply-connected}, namely the existence of simply-connected symplectic Calabi--Yaus with $b_3=0$ and $b_2$ arbitrarily large.

Simple examples of hyperbolic orbifolds can be obtained by doubling right-angled hyperbolic Coxeter polytopes. There is an infinite series of such polytopes in dimension four that can be obtained from the right-angled $120$-cell by taking a chain of hyperbolic 120-cells, each joined to the last across a single face.  Doubling such a polytope produces a hyperbolic orbifold homeomorphic to $S^4$, with each vertex in the polytope giving rise to a $\Z^2_3$-orbifold point. Notice that the singularities of these orbifolds are modeled on the action (\ref{orbiaction}) and so lead to singularities in the twistor space $Z/\Gamma$ which can be resolved by the above blow-up construction. The following result is a refinement of the statement of Theorem \ref{non-Kahler_simply-connected}, where now $b_2$ is given explicitly.

\begin{theorem} Let $\H^4/\Gamma$ be an orbifold obtained by doubling of a right-angled four-dimensional Coxeter hyperbolic polytope $P$. Let $\widetilde{Z}$ be the crepant resolution of the corresponding twistor space $Z/\Gamma$. Then $\pi_1(\widetilde{Z}) = 1$ whilst $b_3(\widetilde{Z})=0$. In particular $\widetilde{Z}$ is not diffeomorphic to a K\"ahler Calabi--Yau manifold. At the same time $b_2(\widetilde{Z})=1 + V(P)+2F(P)$, where $V(P)$ is the number of vertices of $P$ and $F(P)$ is the number of 2-faces.
\end{theorem}

\begin{proof} The proof that $\pi_1(\widetilde{Z}) =1$ is identical to that of Proposition \ref{fundamental-group-prop}. To prove the claims concerning $b_2$ and $b_3$ we first show that $b_{2i+1}(Z/\Gamma)=0$ for $i=0,1,2$ and $b_{2i}(Z/\Gamma)=1$ for $i=0,1,2,3$. To this end, take a finite index torsion free subgroup $\Gamma'$ of $\Gamma$ and consider the quotient manifold $M^4=\H^4/\Gamma'$. Then $Z/\Gamma$ is a quotient of the twistor space $Z/\Gamma'$ of $M$ by the finite group $\Gamma/\Gamma'$. The cohomology of $Z/\Gamma$ can be identified with the invariants of the action of $\Gamma/\Gamma'$ on $H^*(Z/\Gamma')$. At the same time the co-homology of $Z/\Gamma'$ is a free module over $H^*(M)$ with generator $[\omega]\in H^2(Z/\Gamma')$ (where $\omega$ is the symplectic form on $Z'$). Since the quotient of $M$ by $\Gamma/\Gamma'$ is $S^4$ while $[\omega]$ is an invariant of the action, it follows that the ring on invariants of the action of $\Gamma/\Gamma'$ is given by all powers of $[\omega]$.

Now the claims concerning $b_2(\widetilde{Z})$ and $b_3(\widetilde{Z})$ follow from Proposition \ref{b2b3}: the $\Z_2$-singularities of $Z/\Gamma$ form a collection of two-spheres embedded in $Z/\Gamma$; there are in total $V(P)$ spheres projecting to the vertexes of the doubled Coxeter polytope, and $2F(P)$ projecting to the 2-faces.
\end{proof}

\section{Concluding remarks and questions}

We close the article with some remarks on further applications of the techniques used here, together with some natural questions which the construction raises.

Firstly, the construction of the blow-up $\widetilde{Z}$ used no special properties of the orbifold $Z/\Gamma$ beyond the fact that all orbifold groups were abelian. It should adapt in a straightforward fashion to show that any such symplectic orbifold admits a symplectic resolution.

The question of crepancy is more subtle. Here we should limit ourselves to dimension 6. Indeed in algebraic geometry, there is a stark difference between complex dimension 3 and higher dimensions. On the one hand, crepant resolutions are known not always to exist in dimension 4 or higher. On the other hand, in dimension 3 any complex projective variety with quotient singularities (orbifolds) and trivial canonical bundle admits a crepant resolution, i.e.\ also with a trivial canonical bundle. (See \cite{R} and the references given there.) This motivates the following conjecture.

\begin{conjecture}\label{symplectic_crep_res_conjecture}
Every $6$-dimensional symplectic Calabi--Yau orbifold (i.e., with $c_1=0$) admits a symplectic {\it crepant} resolution.
\end{conjecture}

In this article we have shown of course that this conjecture holds for
all the symplectic Calabi--Yau orbifolds coming from Theorem \ref{mainhyper}. Exactly the same techniques should apply to the case of Calabi--Yau orbifolds with abelian singularities; we hope to address this in future work.

Next we turn to the question of whether or not the examples constructed can be symplectomorphic
to algebraic Calabi--Yau varieties. Certainly for the majority of them this is ruled out, because
the fundamental group of an algebraic Calabi--Yau is virtually Abelian. Moreover, since the fundamental group is a birational invariant, this also prevents them from even being birational to algebraic
Calabi--Yau varieties. In general we think that the following is very plausible:

\begin{conjecture}
No symplectic Calabi--Yau manifold arising in the proof of Theorem \ref{main} is symplectically birational to an algebraic Calabi--Yau manifold. Moreover, two symplectic Calabi--Yau manifolds coming from different hyperbolic orbifolds are not symplectically birational.
\end{conjecture}

We would like to finish with some speculation concerning a different type of symplectic 6-manifold, namely those which are Fano, in the sense that $[\omega] = c_1$ in $H^2(\Z)$ (where $[\omega]$ denotes the sypmlectic class). Thanks to Seiberg--Witten theory and the work of Taubes (see, e.g., \cite{taubes}) we know that in dimension 4 there exist only $10$ symplectic Fano manifolds up to symplectomorphism, essentially just the algebraic Fano surfaces (i.e. rational surfaces with ample anti-canonical bundle) with anti-canonical polarization. Moreover, in algebraic geometry, it is known that the number of deformation families of Fano manifolds of given dimension is finite. With this in mind we ask the following question:

\begin{question}\label{symconj}
Is the number of symplectic Fano $6$-manifolds finite up to symplectomorphism?
\end{question}

Note that in dimension $12$ and higher there exists a construction \cite{FP,Re} of symplectic Fanos of arbitrary topological complexity. So if the answer to Question \ref{symconj} is positive this would be a reflection of a certain deep low-dimensional phenomena. Note that a positive answer to this question in full generality would require a genuinely new idea or technique since currently no non-trivial topological restrictions are known for symplectic manifolds of dimension higher than~$4$. At the same time we are unaware of any potential construction that could produce a negative answer. We conjecture that the answer to Question \ref{symconj} is yes when the symplectic $6$-dimensional Fano admits a Hamiltonian $S^1$-action (such manifolds with $b_2=1$ are classified in \cite{mcfano}) or, more generally, when it is assumed to be symplectically uniruled. As a final remark, it would also be interesting to find out the minimal dimension in which there exist symplectic Fanos $M^{2n}$, with \emph{unbounded} volume $c_1(TM)^{n}$. For the moment we know that $4<2n\le 12$.







\end{document}